\documentclass[12pt]{amsart}

\usepackage{amsmath,amssymb,setspace,nicefrac}
\usepackage[active]{srcltx}

\usepackage{amsmath,amssymb,setspace,nicefrac}
\usepackage[active]{srcltx}
\usepackage[pagebackref, colorlinks, linkcolor=red, citecolor=blue, urlcolor=blue, hypertexnames=true]{hyperref}
\usepackage[backrefs]{amsrefs}

\allowdisplaybreaks
\usepackage[all,cmtip]{xy}

\usepackage{color}

\newcommand{\Z}{{\mathbb Z}}
\newcommand{\N}{{\mathbb N}}
\newcommand{\T}{{\mathbb T}}
\newcommand{\F}{{\mathbb F}}

\newtheorem{thm}{Theorem}[section]
\newtheorem{cor}[thm]{Corollary}
\newtheorem{lem}[thm]{Lemma}

\newtheorem{definition}[thm]{Definition}
\newtheorem{example}[thm]{Example}
\newtheorem{remark}[thm]{Remark}

\theoremstyle{definition}


\begin{document}
\title[Topological stability and pseudo-orbit tracing property of group actions]{Topological stability and pseudo-orbit tracing property of group actions }
\author{Nhan-Phu Chung}
\address{Nhan-Phu Chung, Department of Mathematics, Sungkyunkwan University, Suwon 440-746, Korea.} 
\email{phuchung@skku.edu; phuchung82@gmail.com}
\author{Keonhee Lee}
\address{Keonhee Lee, Department of Mathematics, Chungnam National University, Daejeon 305-764, Korea.}
\email{khlee@cnu.ac.kr}
\date{\today}
\maketitle

\begin{abstract}

In this paper we extend the concept of topological stability from homeomorphisms to group actions on compact metric spaces, and prove that if an action of a finitely generated group is expansive and has the pseudo-orbit tracing property then it is topologically stable. This represents a group action version of the  Walter's stability theorem \cite{Walters78}. Moreover we give a class of group actions with topological stability or pseudo-orbit tracing property. On the other hand, we also provide a characterization of subshifts of finite type over finitely generated groups in term of pseudo-orbit tracing property.
\end{abstract}

\onehalfspace
\section{Introduction}

\hspace{2mm}
In 1970, Walters \cite{Walters70} introduced the notion of {\it topological stability}, a kind of stability for homeomorphisms in which continuous pertubations are allowed. In that paper he proved that Anosov diffeomorphisms on compact smooth manifolds are not only structurally stable but also topologically stable. Several results dealing with this new kind of stability were then appearing. For instance, Nitecki \cite{N} proved that Axiom $A$ diffeomorphisms with the strong transversality condition on compact smooth manifolds are topologically stable. Afterwards, Walters \cite{Walters78} proved that every expansive homeomorphisms with  pseudo-orbit tracing property on a compact metric space is topologically stable. Very recently, Lee and Morales \cite{LM} introduced the notions of topological stability and pseudo-orbit tracing property for Borel measures on compact metric spaces, and showed that any expansive measure with pseudo-orbit tracing property is topologically stable.

\hspace{2mm}
In this paper we will obtain a group action version of this result. Indeed, we introduce the notion of topological stability for an action of a finitely generated group on a compact metric space, and prove that if a group action is expansive and has the pseudo-orbit tracing property then it is topologically stable in Scection 2. This represents a further contribution to the study of the pseudo-orbit tracing property (or shadowing) of  group action developed elsewhere in the recent literature \cite{Oprocha2008,OT,Pilyugin,MR2028929}. Furthermore, in Scection 3 we establish a characterization of subshifts of finite type over finitely generated groups via the pseudo-orbit tracing property. This characterization extends the main results of \cite{Walters78} and \cite{Oprocha2008} when the acting group is $\Z$ and $\Z^d$, respectively. Finally, in Section 4 we prove that every equicontinuous action of an infinite, finitely generated group on a Cantor space always has pseudo-orbit tracing property.

\hspace{2mm}
We round out the introduction with some notations that we will use in the paper. Let  $G$ be a finitely generated discrete group and $X$ be a compact metric space with a metric $d$. Put $Homeo(X)$ the space of all homeomorphisms of $X$. We denote by $Act(G,X)$ the set of all continuous actions $T$ of $G$ on $X$. Let $Homeo(X)^{G}=\prod _{G}Homeo(X)$ be the set of homomorphisms from $G$ to $Homeo(X)$ with the product topology. Then $Act(G,X)$ can be considered as a subset of $Homeo(X)^{G}$. Let $A$ be a finitely generating set of $G$. We define a metric $d_A$ on $Act(G,X)$ by $$d_A(T,S):=sup_{\substack{x\in X\\ a\in A}}d(T_ax,S_ax),$$ for $T,S\in Act(G,X)$.

\textbf{Acknowledgement:} N.-P. Chung was supported by the NRF grants funded
by the Korea government (MSIP) (No. NRF-2016R1A5A1008055 and No. NRF-2016R1D1A1B03931922) and K. Lee was supported by the NRF grant funded
by the Korea government (MSIP) (No. NRF-2015R1A2A2A01002437).

\section{Topoological spability and pseudo-orbit tracing property}

\hspace{2mm}
First of all, we introduce the notion of topological stability of a finitely generated group action on a compact metric space.

\begin{definition}
\label{D-topological stability}
Let $A$ be a finitely generating set of $G$, and let $T \in Act(G,X)$. We say that $T$ is \textit{A-topologically stable} if for every $\varepsilon>0$, there exists $\delta>0$ such that if $S$ is another continuous action of $G$ on $X$ with $d_A(T,S)<\delta$ then there exists a continuous map $f:X\to X$ with $T_gf=fS_g$, for every $g\in G$ and $d(f,Id_X):=\sup_{x\in X}d(f(x),x)\leq\varepsilon$.
\end{definition}

\hspace{2mm}
It is clear that the definition of $A$-topological stability of $T$ does not depend on the choice of a compatible metric $d$ on $X$. Note that the definition of topological stability of a homeomorphism introduced in \cite{Walters70} coincides with our definition when $G=\Z$ and $A=\{1\}$. Furthermore, we can see that topological stability of $T$ does not depend on the choice of a symmetric finitely generating set $A$ of $G$. Recall that $A$ is {\it symmetric} if for any $a \in A$, $a^{-1}\in A$.

\begin{lem}
\label{L-topological stability does not depend on the finitely generating sets}
 Let $A$ and $B$ be symmetric finitely generating sets of $G$. For any $T \in Act(G,X)$, $T$ is $A$-topologically stable if and only if it is $B$-topologically stable.
\end{lem}

\begin{proof}
Suppose $T$ is $A$-topologically stable. Then for any $\varepsilon >0$, there exists $\delta'>0$ such that if $S$ is another continuous action of $G$ on $X$ with $d_A(T,S)<\delta'$ then there exists a continuous map $f:X\to X$ with $T_gf=fS_g$, for every $g\in G$ and $d(f,Id_X)\leq\varepsilon$. It suffices to show that there exists $\delta>0$ such that for any $S\in Act(G,X)$, if $d_B(T,S)<\delta$ then $d_A(T,S)<\delta'$. Put $m:=\max_{a\in A}\ell_B(a)$, where $\ell_B$ is the word length metric on $G$ induced by $B$. Choose $\delta_1>0$ such that $m\delta_1<\delta'$. Since $X$ is compact, $A$ and $B$ are finite and the action $T$ is continuous, there exists $\delta>0$ such that $d(T_hx,T_hy)<\delta_1$ for $x,y\in X$ with $d(x,y)<\delta$ and for $h\in G$ with $\ell_B(h)\leq m$.
 For any $a\in A$, we write $a$ as $b_1\cdots b_{\ell(a)}$, where $\ell(a)=\ell_B(a)\leq m$, $b_i\in B, i=1,\cdots, \ell(a)$. Then for any $S\in Act(G,X)$ with $d_B(T,S)<\delta$, we have
 \begin{eqnarray*}
d(T_ax,S_ax)&=&d(T_{b_1\cdots b_{\ell(a)}}x,S_{b_1\cdots b_{\ell(a)}}x)\\
&\leq &d(T_{b_1\cdots b_{\ell(a)-1}}T_{b_{\ell(a)}}x,T_{b_1b_2\cdots b_{\ell(a)-1}}S_{b_{\ell(a)}}x)\\
&+&d(T_{b_1b_2\cdots b_{\ell(a)-2}}T_{b_{\ell(a)-1}}S_{b_{\ell(a)}}x,T_{b_1b_2\cdots b_{\ell(a)-2}}S_{b_{\ell(a)-1}}S_{b_{\ell(a)}}x)\\
&+&\cdots +d(T_{b_1}T_{b_2}S_{b_3\cdots b_{\ell(a)-1}b_{\ell(a)}}x, T_{b_1}S_{b_2}S_{b_3\cdots b_{\ell(a)-1}b_{\ell(a)}}x)\\
&+&d(T_{b_1}S_{b_2\cdots b_{\ell(a)-1}b_{\ell(a)}}x,S_{b_1\cdots b_{\ell(a)}}x)\\
&<&m\delta_1<\delta'.
\end{eqnarray*}
This means that $d_{A}(T,S)<\delta'$, and so completes the proof.
\end{proof}

\begin{definition}
An action $T\in Act(G,X)$ is said to be \textit{topologically stable} if it is $A$-topologically stable for a symmetric finitely generating set $A$ of $G$.
\end{definition}

\begin{remark}
Let $A$ be a finitely generating set of $G$. We define a metric $\widetilde{d}_A$ on $Act(G,X)$ by $$\widetilde{d}_A(T,S):=sup_{\substack{x\in X\\ a\in A}}\{d(T_ax,S_ax)+d(T_a^{-1}x,S_a^{-1}x)\},$$ for $T,S\in Act(G,X)$. Clearly $d_A$ and $\widetilde{d}_A$ are equivalent. Furthermore, the topology on $Act(G,X)$ induced by $\widetilde{d}_A$ coincides with the product topology on $Act(G,X)$ inherited from $Homeo(X)^G$. Hence the space $Act(G,X)$ is a separable complete metrizable topological space, and so a Polish space.
\end{remark}

\hspace{2mm}
If $T$ and $S$ are two continuous actions of $G$ on X with $d_A(T,S)<\delta$, then the $S$-orbit $\{S_gx\}$ of $x\in X$ is nearly a $T$-orbit in the sense that $d(T_aS_gx,S_{ag}x)<\delta$ for all $a\in A$ and $g\in G$. This observation motivates the following definition.

\begin{definition}
Let $A$ be a finitely generating set of $G$ and $\delta >0$. A \textit{$\delta$ pseudo-orbit} of $T \in Act(G,X)$ with respect to $A$ is a sequence $\{x_g\}_{g\in G}$ in $X$ such that $d(T_ax_g,x_{ag})<\delta$ for all $a\in A,g\in G$.
\end{definition}

\begin{definition}
\label{D-POTP}
Let $A$ be a finitely generating set of $G$. An action $T \in Act(G,X)$ is said to have the \textit{pseudo-orbit tracing property (abbrev. POTP)} with respect to $A$ if for every $\varepsilon>0$, there exists $\delta>0$ such that any $\delta$ pseudo-orbit $\{x_g\}_{g\in G}$ for $T$ with respect to $A$ is $\varepsilon$-traced by some point $x$ of $X$, that is, $d(T_gx, x_g)<\varepsilon$ for all $g\in G$.
\end{definition}

\hspace{2mm}
Note that the pseudo-orbit tracing property of $T$ does not depend on the choice of a compatible metric $d$ of $X$. Osipov and Tikhomirov \cite{OT} introduced the notion of  POTP which they called shadowing for actions of finitely generated groups by using symmetric finitely generating sets of the acting groups.
They showed that the definition of POTP does not depend on the choice of symmetric finitely generating sets. Indeed, it is not hard to check that it actually does not depend on the choice of general finitely generating sets, see for example \cite[Lemma 2.2]{Pilyugin}.

\begin{definition}
We say that an action $T \in Act(G,X)$ has POTP if $T$ has POTP with respect to $A$ for a finitely generating set $A$ of $G$.
\end{definition}

\hspace{2mm}
Note that the definition of POTP of a homeomorphism coincides with our definition when $G=\Z$ and $A=\{1\}$.
We recall that an action $T$ of a group $G$ on a compact metric space $(X,d)$ is {\it expansive} if there exists an open subset $U$ in $X\times X$ such that $\Delta_X=\bigcap_{g\in G} g^{-1}U$, where $\Delta_X:=\{(x,x):x\in X\}$ and the action of $G$ on $X \times X$ is defined by $g(x,y):=(gx,gy)$ for every $x,y\in X,g\in G$. Note that $T$ is expansive if and only if there exists a constant $c>0$ called an \textit{expansive constant} of $T$  such that for every $x\neq y$, one has $\sup_{g\in G}d(gx,gy)>c$. 
Now we prove that if a finitely generated group action is expansive and has POTP then it is topologically stable. This extends the main results of \cite[Theorems 4 and 5]{Walters78} to group actions.

\begin{thm}
\label{T-expansiveness and POTP imply topological stability}
If an action $T \in Act(G,X)$ is expansive and has POTP, then it is topologically stable. Moreover, for a finitely generating set $A$ of $G$, for $\varepsilon>0$ with $\varepsilon<\eta,$ where $\eta$ is an expansive constant of $T$, there exists $\delta>0$ such that if $S$ is another continuous action of $G$ on $X$ with $d_A(T,S)<\delta$ then there is a unique map $f:X\to X$ with $T_gf=fS_g$ for every $g\in G$ and $d(f,Id_X)\leq\varepsilon$. Furthermore, if $S$ is also expansive with an expansive constant $\eta_S\geq 2\varepsilon$, then the conjugating map $f$ is injective.
\end{thm}

\hspace{2mm}
To prove the above theorem, we need following two lemmas.

\begin{lem}
\label{L-expansive and POTP imply uniqueness of trace}
Let $T \in Act(G,X)$ be an expansive action with POTP with respect to a symmetic finitely generating set $A$ of $G$. Let $\varepsilon<\eta/2$ and $\delta$ corresponds to $\varepsilon$ as in Definition \ref{D-POTP}, where $\eta$ is an expansive constant of $T$. Then every $\delta$ pseudo-obrit of $T$  is $\varepsilon$-traced by a unique point in $X$.
\end{lem}

\begin{proof}
Let $\{x_g\}_{g\in G}$ be a $\delta$ pseudo-orbit of $T,$ and let $x,y$ be two points that $\varepsilon$-trace $\{x_g\}_{g\in G}$. Then one has $d(T_gx,T_gy)\leq d(T_gx,x_g)+d(x_g,T_gy)<2\varepsilon<\eta$ for every $g\in G$. By the expansiveness of $T$, we get $x=y$.
\end{proof}

\begin{lem}
\label{L-expansiveness determines the topology}
Let $T \in Act(G,x)$ be an expansive action with an expansive constant $\eta$. Then, for any $\varepsilon>0$, there exists a non-empty finite subset $F$ of $G$ such that whenever $\sup_{g\in F}d(T_gx,T_gy)\leq \eta$, we have $d(x,y)<\varepsilon$.
\end{lem}

\begin{proof}
Assume that there exists $\varepsilon>0$ such that for any non-empty finite subset $F$ of $G$, there exist $x_F,y_F\in X$ such that $$\sup_{g\in F} d(T_gx,T_gy)\leq \eta \mbox{ and  }d(x_F,y_F)\geq\varepsilon.$$ Choose a sequence of finite subsets $F_n$ of $G$ satisfying 
$$ F_1\subset F_2\subset \cdots  \mbox{  and  }G=\bigcup_{n\in \N}F_n.$$ 
Then for every $n\in \N$, there exist $x_n,y_n\in F_n$ such that 
$$\sup_{g\in F_n} d(T_gx_n,T_gy_n)\leq \eta \mbox{  and  } d(x_n,y_n)\geq\varepsilon.$$
After taking a subsequence, we can assume that $x_n\to x$ and $y_n\to y$. Then we have $d(T_gx,T_gy)\leq \eta$ for all $g\in G$ and $d(x,y)\geq\varepsilon$, which contradicts with the expansiveness of $T$.
\end{proof}

\begin{proof}
[Proof of Theorem \ref{T-expansiveness and POTP imply topological stability}]
Let $\eta$ be an expansive constant of $T$ and let $\varepsilon<\eta/3$. Let $A$ be a finitely generating set of $G$. Choose $\delta$ corresponding to $\varepsilon$ as in Definition \ref{D-POTP}. Let $S$ be a continuous action of $G$ on $X$ with $d_A(T,S)<\delta$. For any $x \in X$, we see that the $S$-orbit $\{S_gx\}_{g\in G}$ of $x$ is a $\delta$ pseudo-orbit for $T$. By Lemma \ref{L-expansive and POTP imply uniqueness of trace}, there is a unique point denoted by $f(x)$ whose $T$-orbit $\varepsilon$ traces $\{S_gx\}_{g\in G}$. Then we have the map $f:X\to X$ satisfying $$d(T_gf(x),S_gx)<\varepsilon \mbox{ for all } g\in G, x\in X \mbox{   (*)  }.$$ In particular, we have $d(f(x),x)<\varepsilon$ for every $x\in X$, and hence $d(f,Id_X)\leq\varepsilon$.

\hspace{2mm}
 Now we will prove that $T_gf(x)=fS_g(x)$ for every $x\in X,g\in G$. In fact, we have
 $$d(T_hf(S_gx),S_{hg}x)=d(T_hf(S_gx),S_hS_gx)<\varepsilon$$
 for $x \in X$ and $g,h \in G$. On the other hand, applying (*) again, we obtain
 $$d(T_hT_gf(x),S_{hg}x)=d(T_{hg}f(x),S_{hg}x)<\varepsilon.$$
Then we get $T_gf(x)=fS_g(x)$  by Lemma \ref{L-expansive and POTP imply uniqueness of trace}.

\hspace{2mm}
Next we will show that $f$ is continuous. Let $\varepsilon_1>0$. By Lemma \ref{L-expansiveness determines the topology}, there exists a non-empty finite subset $F$ of $G$ such that whenever $\sup_{g\in F}d(T_gx,T_gy)\leq \eta$ one has $d(x,y)<\varepsilon_1$. Choose $\delta_1>0$ such that for every $x,y\in X$ with $d(x,y)<\delta_1$, one has $d(S_gx,S_gy)<\eta/3$ for every $g\in F$. Then, for any $x,y\in X$ with $d(x,y)<\delta_1$ and $g \in F$, we get 
\begin{eqnarray*}
d(T_gf(x),T_gf(y))&=&d(fS_g(x),fS_g(y))\\
&\leq & d(fS_g(x),S_g(x))+d(S_g(x),S_g(y))+d(S_g(y),fS_g(y))\\
&<& \varepsilon+\eta/3+\varepsilon<\eta.
\end{eqnarray*}
Thus $d(f(x),f(y))<\varepsilon_1$ for $x,y\in X$ with $d(x,y)<\delta_1$. This mean that $f$ is continuous. 

\hspace{2mm}
Assume that there is another map $f_1$ such that $T_gf_1=f_1S_g$ for all $g\in G$ and $d(f_1,Id_X)\leq\varepsilon$. Then, for any $g\in G$, we have

\begin{eqnarray*}
d(T_gf(x),T_gf_1(x))&\leq & d(T_gf(x),S_g(x))+d(S_g(x),T_gf_1(x)) \\
&=&d(fS_g(x),S_g(x))+d(S_g(x),f_1S_g(x))<2\varepsilon<\eta.
\end{eqnarray*}
By the expansiveness of $T$, we get $f(x)=f_1(x)$. Finally, we will prove the last assertion. Let $f(x)=f(y)$. Then for any $g\in G$, we have 
\begin{eqnarray*}
d(S_gx,S_gy)&\leq & d(S_gx,f(S_gx))+d(f(S_gx),f(S_gy))+d(f(S_gy),S_gy)\\
&=&d(S_gx,f(S_gx))+d(T_gf(x),T_gf(y))+d(f(S_gy),S_gy)
<2\varepsilon\leq \eta_S,
\end{eqnarray*}
and so $x=y$ by the expansiveness of $S$.
\end{proof}

\hspace{2mm}
Next we will provide a class of topologically stable actions. First, we recall the definition of nilpotent group. Let $G$ be a countable group. The \textit{lower central series} of $G$ is the sequence $\{G_i\}_{i\geq 0}$ of subgroups of $G$ defined by $G_0=G$ and $G_{i+1}=[G_i,G]$, where $[G_i,G]$ is the subgroup of $G$ generated by all commutators $[a,b]:=aba^{-1}b^{-1}$, $a\in G_i,b\in G$. The group $G$ is said to be \textit{nilpotent} if there exists $n\geq 0$ such that $G_n=\{e_G\}$. The such smallest $n$ is called the \textit{nilpotent degree} of $G$. 

\begin{thm}
\label{T-POTP of actions of nilpotent groups}
Let $G$ be a finitely generated virtually nilpotent group, i.e. there exists a nilpotent subgroup $H$ of $G$ with finite index. Let $T$ be a continuous action of $G$ on $X$. If there exists an element $g\in G$ such that $T_g$ is expansive and has POTP, then $T$ is topologically stable.
\end{thm}

\hspace{2mm}
The following two lemmas were proved implicitly in (\cite[Lemmas 1 and 2]{OT}). For convenience, we provide those implicit proofs here.

\begin{lem}
\label{L-expansiveness+POTP for normal subgroups implies POTP}
Let $G$ be a finitely generated group and $H$ be a finitely generated normal subgroup of $G$. Let $T$ be a continuous action of $G$ on $X$. If the restriction action $T_H$ of $T$ to $H$ is expansive and has POTP then $T$ has POTP. 
\end{lem}

\begin{proof}
Let $A$ be a symmetric finitely generating set of $H$. We can add more elements to $A$ to get a symmetric finitely generating set $B$ of $G$. Let $c$ be an expansive constant of $T_H$. Since $X$ is compact and $B$ is finite, there exists $0<\eta<c/3$ such that $d(T_bx,T_by)<c/3$ for every $b\in B$ and every $x,y\in X$ with $d(x,y)<\eta$. Let $\varepsilon>0$ be a constant with $\varepsilon<\eta$. We can choose $0<\delta<\varepsilon$ such that every $\delta$ pseudo-orbit for $T_H$ with respect to $A$ is $\varepsilon$-traced by some point of $X$. Let $\{x_g\}_{g\in G}$ be a $\delta$ pseudo--orbit of $T$ with respect to $B$. For every $g\in G$, the sequence $\{x_{hg}\}_{h\in H}$ is a $\delta$ pseudo-orbit of $T_H$ with respect to $A$. Since $T_H$ is expansive, by Lemma \ref{L-expansive and POTP imply uniqueness of trace}, there exists a unique point $y_g\in X$ such that 
\begin{align}
\label{F-expansiveness on normal subgroups}
d(x_{hg},T_hy_g)<\varepsilon, \mbox { for every } h\in H.
\end{align} 
Now we will prove that $y_g=T_gy_e$ for $g\in G$. Fix $g\in G$ and $b\in B$. For any $h\in H$, there exists $h'\in H$ such that $hb=bh'$. Then we have  $$d(x_{bh'g},T_hy_{bg})=d(x_{hbg},T_hy_{bg})<\varepsilon \mbox{  and  } d(T_bx_{h'g},T_bT_{h'}y_g)<c/3.$$ Hence we get
$$d(T_hy_{bg}, T_hT_by_g)\leq d(T_hy_{bg},x_{bh'g})+d(x_{bh'g}, T_bx_{h'g})+d(T_bx_{h'g},T_{bh'}y_g)<c.$$
Since $T_H$ is expansive, we have $T_by_g=y_{bg}$ for every $b\in B$. As $B$ is a symmetric generating set of $G$, we get $T_gy_e=y_g$ for every $g\in G$. On the other hand, by applying $h=e_G$ for (\ref{F-expansiveness on normal subgroups}), we have $d(x_g,y_g)<\varepsilon$ for every $g\in G$ and hence $d(x_g,T_gy_e)<\varepsilon$ for $g\in G$.
\end{proof}

\begin{lem}
\label{L-POTP for actions of nilpotent groups}
Let $G$ be a finitely generated nilpotent group and $T$ be a continuous action of $G$ on $X$. If there exists $g\in G$ such that $T_g$ is expansive and has POTP then $T$ has POTP.
\end{lem}

\begin{proof}
We prove by induction on the nilpotent degree $n$ of $G$. If $n=1$ then the group $G$ is abelian and hence $H=\langle g\rangle$ is a normal subgroup of $G$. Thus, $T$ has POTP by Lemma \ref{L-expansiveness+POTP for normal subgroups implies POTP}. Let $n>1$ and assume that the statement of the lemma is true for all nilpotent groups with nilpotent degree is less than or equal to $n-1$. Put $G_1:=[G,G]$ and $K:=\langle G_1,g\rangle$. Then $K$ has the nilpotent degree at most $n-1$ \cite[Proposition 2]{OT}. It is known that $G_1$ is finitely generated \cite[Lemma 6.8.4]{CC2010} and hence $K$ is finitely generated. Thus, from the induction assumption, we know $T_K$ has POTP. Since $T_g$ is expansive and $g\in K$, we have $T_K$ is expansive. As $K$ is a normal subgroup of $G$ \cite[Proposition 2]{OT}, we complete the proof by Lemma \ref{L-expansiveness+POTP for normal subgroups implies POTP}.
\end{proof}

\begin{proof}[Proof of Theorem \ref{T-POTP of actions of nilpotent groups}]
Let $H$ be a nilpotent normal subgroup of $G$ with finite index. Then $H$ is finitely generated \cite[Proposition 6.6.2]{CC2010}. Since $H$ has finite index in $G$, there exists $n\in\N$ such that $g^n\in H$. Because $T_g$ is expansive and has POTP, $T_{g^n}=T_g^n$ is also expansive and has POTP. Thus, from Lemma \ref{L-POTP for actions of nilpotent groups}, the action $T_H$ has POTP. Since $g^n\in H$ and $T_{g^n}$ is expansive, one has $T_H$ is expansive. Applying Theorem \ref{T-expansiveness and POTP imply topological stability} and Lemma \ref{L-expansiveness+POTP for normal subgroups implies POTP}, we complete the proof.
\end{proof}

\hspace{2mm}
In the following example given in  \cite[Example 1.1]{HSW}, we see that the integral Heisenberg group $H$ induces an action on the torus $\T^{3n}$ which is topologically stable.

\begin{example}
 Let $H$ be the integral Heisenberg group, i.e. $H=\langle a,b,c|ac=ca,bc=cb, ab=bac\rangle$. Then $H$ is a nilpotent group. For every $n\in N$, we let $x,y\in SL(n,\Z)$ be such that $xy=yx$. Put 
\begin{align*}
a= \begin{pmatrix}
    x & I_n & 0 \\
    0 & x & 0 \\
    0 & 0 & x
\end{pmatrix},
\mbox{  } b=\begin{pmatrix}
y & 0 & 0\\
0 & y & I_n\\
0 & 0 & y
\end{pmatrix},
\mbox{  }c=\begin{pmatrix}
I_n & 0 & x^{-1}y^{-1}\\
0& I_n & 0\\
0& 0 & I_n
\end{pmatrix}.
\end{align*}

Then $a,b,c$ satisfy the relations in $H$. Let $T$ be the natural action of $H \leq SL(3n,\Z)$ on $\T^{3n}$. If $x$ has no eigenvalues of modulus 1 then $a$ also has no eigenvalues of modulus 1. Then we see that $T_a$ is expansive and has POTP. Applying Theorems \ref{T-expansiveness and POTP imply topological stability} and  \ref{T-POTP of actions of nilpotent groups}, we know that the action $T$ is topologically stable. Similarly for the case $y$ has no eigenvalues of modulus 1.
\end{example}

\section{Shifts of finite type and POTP}

\hspace{2mm}
Let $A$ be a finitely generating set of $G$. For any $k\in \N$, we put $B(k):=\{g\in G:\ell_A(g)\leq k\}$. Let $S$ be a nonempty  finite set. We denote by $S^G$ the product space $\prod_G S$ endowed with the product topology. We consider the action of $G$ on $S^G$ by right shifts, i.e. $(gx)_h=x_{hg}$ for every $g,h\in G$ and $x\in S^G$. For any nonempty finite subset $F$ of $G$ and $x\in S^G$, we denote by $x_F$ the restriction of $x$ to $F$ and $\pi_F: S^G\to S^F$ the natural projection map. An element $f\in S^F$ is called a \textit{pattern}. If $F=\{g\}$ for some $g\in G$, we write it simply as $x_g$. A closed $G$-invariant subset $X$ of $S^G$ is called a \textit{subshift}. A pattern $f\in S^F$ is said to be \textit{allowed} for the shift $X$ if there exists $x$ in $X$ such that $f=x_F$. For every $k\in \N$, a pattern $f\in S^{B(k)}$ is called a \textit{k-block}, and we denote by $B_k(S)$ and $B_k(X)$ the set of all $k$-blocks and the set of all $k$-blocks allowed in $X$, respectively. Then, we  put $B(S)=\bigcup_{k\in \N}B_k(S)$ and $B(X)=\bigcup_{k\in \N}B_k(X)$. Let $W$ be a set of patterns. We define 
$$X_W:=\{x\in S^G: (gx)_F=x_{Fg}\notin W \mbox{ for all } g\in G \mbox{ and all finite subset  } F (\neq \emptyset ) \mbox{  of  }  G\}.$$

A subshift $X$ of $S^G$ is said to be of \textit{finite type} if there is a nonempty finite subset $F$ of $G$ and a subset $P$ of $S^F$ such that $$X=\{x\in S^G:\pi_F(gx)\in P, \forall g\in G\}=:S(F,P).$$ 
The subsets $F$ and $P$ are called a \textit{defining window} and a \textit{set of allowed words} for $X$, respectively.
It is clear that if $X$ is a subshift of finite type defined by a defining window $F$ and a set of allowed words $P$, then for any $F'\supset F$, it is also a subshift of finite type defined by $F'$ and $P'=\{x\in S^{F'}:\pi_F(x)\in P\}$. The following lemma is clear from the definitions of subshift of finite type and $X_W$.

\begin{lem}
\label{L-SFT coincides with subshift with finite forbidden words}
Let $X$ be a subset of $S^G$. Then $X$ is a subshift of finite type if and only if there exist a nonempty finite subset $F$ of $G$ and $W\subset S^F$ such that $X=X_W$.
\end{lem}

\hspace{2mm}
Now we establish a characterization of subshifts of finite type over finitely generated groups via POTP. This characterization extends the main results of \cite{Walters78} and \cite{Oprocha2008} when the acting group is $\Z$ and $\Z^d$, respectively. 

\begin{thm}
Let $G$ be a finitely generated group and $S$ be a nonempty finite set. Let $T$ be the right action of $G$ on $S^G$ and $X$ be a subshift of $S^{G}$. Then $X$ is  of finite type if and only if it has POTP, $i.e.$ $T$ has POTP on $X$.
\end{thm}

\begin{proof}
Let $A$ be a symmetric finitely generating set of $G$. We define a metric $d$ on $X$ by 
$$d(x,y)=2^{-k},~~ k:=\sup\{j\in \N:x_g=y_g, \mbox{ for all } g\in B(j)\}$$ 
for $x,y\in X$, where $B(j)=\{g \in G : \ell_{A}(g)\leq j \}$. Then $d$ is a compatible metric on $X$. Assume that $X$ is a subshift of finite type. Then we can choose $M\in \N$ such that $F=B(M)$ is a defining window for $X$. Let $W$ be the set of allowed words for $X$ with respect to the window $F$.
Let $\varepsilon>0$. Choose $m\in\N$ such that $2^{-m}<\varepsilon$ and $m>M$. Put $\delta=2^{-(m+1)}$. Then we will prove that every $\delta$ pseudo-orbit $\{x^{(g)}\}_{g\in G}$ for $T$ on $X$ with respect to $A$ will be $\varepsilon$-traced by some point $x\in X$. Let $\{x^{(g)}\}_{g\in G}$ be a $\delta$ pseudo-orbit for $T$ on $X$ with respect to $A$, i.e. $d(T_ax^{(g)}, x^{(ag)})<2^{-(m+1)}$ for every $g\in G,a\in A$. Then for any $h\in B(m), g\in G$ and $a\in A$, we have $$x^{(g)}_{ha}=(T_ax^{(g)})_h=x_h^{(ag)} \mbox{ (**)}.$$ 
Fix $h\in B(m)$. Then we can write $h$ as $a_1\cdots a_n$, where $a_i\in A$ for every $1\leq i\leq n$ and $n\leq m$. Put $x:=(x_g)_{g\in G}$, where $x_g:=x^{(g)}_{e}$ for every $g\in G$. Applying (**), one has $$(T_gx)_h=x_{hg}=x_e^{(hg)}=x_e^{(a_1\cdots a_ng)}=x_{a_1}^{(a_2\cdots a_ng)}=\cdots =x_{a_1\cdots a_n}^{(g)}=x_h^{(g)},$$ for every $g\in G$. Thus $d(T_gx,x^{(g)})\leq 2^{-m}<\varepsilon$ for every $g\in G$. Now we show that $x\in X$. Since $F=B(M)\subset B(m)$, we have $(T_gx)_f=x_f^{(g)}$  for $f\in F$ and $g\in G$. Hence we get $$\pi_F(T_gx)=\pi_F(x^{(g)})\in W$$ for every $g
\in G$. This implies $x\in X$. 

\hspace{2mm}
Now we will prove the converse. Assume that $X$ has POTP. Take $\delta>0$ such that every $\delta$ pseudo-orbit for $T$ on $X$ with respect to $A$ is $1/2$-traced. Choose $m\in\N$ such that $2^{-m}<\delta$. Let $W\subset B_{m+1}(S)\setminus B_{m+1}(X)$. Then from the definition of $X_W$, one has $X\subset X_W$. Now we claim that $X_W\subset X$. Let $y\in X_W$. Then for every $g\in G$, one has $(gy)_{B(m+1)}\in B_{m+1}(X)$ and hence there exists $x^{(g)}\in X$ such that $x^{(g)}_{B(m+1)}=(gy)_{B(m+1)}$. Then for any $a\in A$ and $g\in G$, we get $$(ax^{(g)})_{B(m)}=x^{(g)}_{B(m)a}=y_{B(m)ag}=x^{(ag)}_{B(m)}$$ because $B(m)a\subset B(m+1)$. Therefore, we have $d(ax^{(g)},x^{(ag)})\leq 2^{-m}<\delta$ and hence $\{x^{(g)}\}_{g\in G}$ is a $\delta$ pseudo-orbit for $T$ on $X$ with respect to $A$. As $X$ has POTP, there exists $x\in X$ such that $\{x^{(g)}\}_{g\in G}$ is $1/2$-traced by $x$, i.e. $d(gx,x^{(g)})<1/2$ for every $g\in G$. Hence we get 
$$x_{B(1)g}=x^{(g)}_{B(1)}=y_{B(1)g}$$ 
for $g\in G$, and thus $y=x\in X$. Applying Lemma \ref{L-SFT coincides with subshift with finite forbidden words}, we complete the proof.
\end{proof}
\section{Equicontinuous actions and POTP}

\hspace{2mm}
We say that an action $T \in Act(G,X)$ is \textit{equicontinuous} if for every $\varepsilon>0$, there exists $\delta>0$ such that for any $x,y\in X$ with $d(x,y)<\delta$, one has $d(gx,gy)<\varepsilon$ for every $g\in G$. $T$ is called \textit{distal} if for every $x\neq y\in X$, one has $\inf_{g\in G}d(gx,gy)>0$.

\hspace{2mm}
It is known that every compact, totally disconnected, metrizable space without isolated points is homeomorphic to the Cantor set and every equicontinuous action of $\Z$ on the Cantor space has POTP. In the following theorem, we can see that it still holds for actions of general groups.

\begin{thm}
\label{L-equicontinuity systems of Cantor spaces have POTP}
Let $G$ be an infinite, finitely generated group and let $T$ be an equicontinuous action of $G$ on the Cantor space $X$. Then $T$ has POTP.
\end{thm}

\begin{proof}
Since $\{0,1\}^G$ is a compact, totally disconnected, metrizable space without isolated points, we can assume that $X=\{0,1\}^G$. Let $A$ be a symmetric finitely generating set of $G$. We enumerate elements of $G$ by $e=g_0,g_1,\cdots$ and define a compatible metric $d$ on $X$ by $$d(x,y):=\sum_{i=1}^\infty\frac{1}{2^i}d_S(x_{g_i},y_{g_i})$$ for $x,y\in X$, where $d_S$ is the metric on $\{0,1\}$ defined by $d_S(a,b)$ is 1 if $a\neq b$ and 0 otherwise. Put $F_n:=\{g_0,g_1,\cdots, g_n\}$. Let $\varepsilon>0$. Choose $m\in\N$ such that $2^{-m}<\varepsilon$. Since the action $T$ is equicontinuous, there exists $k\in \N$ such that whenever $d(x,y)<2^{-k}$, one has $d(gx,gy)<2^{-m}$ for every $g\in G$. This means that for any $x,y\in X$ with $d(x,y)<2^{-k}$, we have $(gx)_{F_m}=(gy)_{F_m}$ for $g\in G$. Let $\{x^{(h)}\}_{h\in G}$ be a $2^{-k}$ pseudo-orbit for $T$ with respect to $A$, i.e. for $a\in A, h\in G$, we have $d(ax^{(h)},x^{(ah)})<2^{-k}$. Then $(gax^{(h)})_{F_m}=(gx^{(ah)})_{F_m}$ for every $a\in A$ and $h,g\in G$. For $g\in G$, we write $g$ as $a_1\cdots a_n$ for some $n\in\N$ and $a_i\in A$. Then we have
\begin{eqnarray*}
(gx^{(e)})_{F_m}&=&(a_1\cdots a_nx^{(e)})_{F_m}=(a_1\cdots a_{n-1}x^{(a_n)})_{F_m}\\
&=&(a_1\cdots a_{n-2}x^{(a_{n-1}a_n)})_{F_m}=\cdots=x^{(a_1\cdots a_n)}_{F_m}=x^{(g)}_{F_m}.
\end{eqnarray*}

Consequently we get $d(gx^{(e)},x^{(g)})<2^{-m}<\varepsilon$, and so complete the proof.
\end{proof}

\begin{cor}
Every distal action of a finitely generated group on the Cantor space has POTP.
\end{cor}

\begin{proof}
From \cite[Corollary 1.9]{AGW}, we know that every distal action of a finitely generated group on the Cantor space is equicontinuous. 
\end{proof}

\begin{remark}
We can not drop out the assumption of Cantor space in Theorem \ref{L-equicontinuity systems of Cantor spaces have POTP} because even in the case $G=\Z$, every distal homeomorphism on a connected space does not have POTP \cite[Theorem 2.3.2]{AH}. 
\end{remark}

Let $G$ be a countable group. A \textit{chain} of $G$ is a sequence of subgroups $\{G_n\}_{n\geq 0}$ such that $G=G_0\geq G_1\geq G_2\geq \cdots$ such that $[G:G_n]<\infty$ for every $n\geq 0$. For a chain $\{G_n\}_{n\geq 0}$, we define tree structure $T (G, \{G_n\})$ as follows. The vertices are $\{g G_n: n\geq 0, g\in G\}$ and $(g_1 G_n, g_2 G_m)$ is an edge if $m= n+ 1$ and $g_2 G_m\subset g_1 G_n$. The boundary $\overleftarrow{G}_{\{G_n\}}$ of $T (G, \{G_n\})$ consists of all sequences $(x_0, x_1, \cdots)$ of vertices with $x_n$ adjacent to $x_{n+ 1}$ for each $n\in \mathbb{Z}_+$. Then $\overleftarrow{G}_{\{G_n\}}$ is a compact metrizable space endowed with the topology generated by the open basis consisting of all subsets $O_x= \{(x_0, x_1, \cdots)\in \overleftarrow{G}_{\{G_n\}}: x_N= x\}$ with $x\in G/G_N$ and $N\in \mathbb{Z}_+$. The natural left actions of $G$ on $G/G_n$ induce the \emph{profinite action} $(\overleftarrow{G}_{\{G_n\}}, G)$, an action of $G$ on $\overleftarrow{G}_{\{G_n\}}$ by homeomorphisms.
 In this case the profinite action $(\overleftarrow{G}_{\{G_n\}}, G)$ is equicontinuous, since for any $x= g G_n\in G/G_n$ and $h\in G$ we have $h O_x= O_y$ with $y= h g G_n\in G/G_n$.
 
 Profinite actions have been studied extensively in rank gradient, orbit equivalence, operator algebras, and sofic entropy theory \cite{AN,chungzhang,Ioana}.
 Since every profinite action of $G$ on $\overleftarrow{G}_{\{G_n\}}$ is equicontinuous and every infinite boundary space $\overleftarrow{G}_{\{G_n\}}$ is a Cantor space, from Theorem \ref{L-equicontinuity systems of Cantor spaces have POTP} we get the following corollary.
 \begin{cor}
Every profinite action of $G$ on an infinite boundary space $\overleftarrow{G}_{\{G_n\}}$ has POTP.
 \end{cor}

\end{document}